\documentclass[11pt,reqno]{amsart}

\usepackage{times}
\usepackage{graphicx}
\usepackage{amssymb}
\usepackage{latexsym}
\usepackage{enumerate}
\usepackage[bookmarks=false]{hyperref}
\calclayout
\theoremstyle{plain}
\numberwithin{equation}{section}
\newtheorem{thm}{Theorem}[section]

\newtheorem{proposition}[thm]{Proposition}

\newtheorem{construction}[thm]{Construction}
\newtheorem{problem}[thm]{Problem}

\begin{document}

\title{A Conic Section Problem Involving the Maximum Generalized Golden Right Triangle}
\author{Jun Li}
\address{
School of Science\\
Jiangxi University of Science and Technology\\ Ganzhou\\
341000\\
China.
}
\email{junli323@163.com}

\begin{abstract}
An interesting conic section problem involving the maximum generalized golden right triangle $T_2$ is solved, and two simple constructions of $T_2$ are shown.
\end{abstract}

\date{2016.6}
\maketitle

\section{Introduction}\label{intro}
As the great astronomer Johannes Kepler stated,``Geometry has two great treasures: one is the theorem of Pythagoras; the other, the division of a line into extreme and mean ratio. The first we may compare to a measure of gold; the second we may name a precious jewel''\cite[p. 160]{cte1}.

\section{A conic section problem}
\label{sec:1}
\begin{figure}[ht]
\centering
\includegraphics[width=1\textwidth]{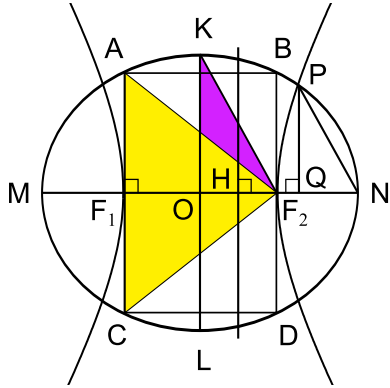}
\caption{An interesting conic section problem}
\label{fig:fg1}
\end{figure}

Let's consider an interesting problem involving an ellipse and a hyperbola in Figure \ref{fig:fg1}. First, we construct an ellipse $\frac{x^2}{a^2} + \frac{y^2}{b^2} = 1$ and a hyperbola
$\frac{x^2}{c^2} - \frac{y^2}{b^2} = 1$, where $a$ is the semi-major axis, $b$ is the semi-minor axis of the ellipse, and $a^2=b^2+c^2$, such that, the eccentricity $e_1$ of the ellipse and $e_2$ of the hyperbola satisfy the condition $e_1 e_2 = 1$, and the foci of the ellipse becomes the corresponding vertex of the hyperbola, next, let $F_1$ and $F_2$ denote the foci of the ellipse, $KL$ the minor axis, $MN$ the major axis, and $O$ the origin, without loss of generality, we set $c=OF_2=1$.

Then, let $P$ be the top-right intersection point of the ellipse and the hyperbola,
construct a segment $PQ$ perpendicular to $ON$ and intersecting $ON$ at the foot $Q$, let H be the intersection point of $ON$ and the right directrix $x = \frac{1}{a}$ of the hyperbola, now, our problem is:
\begin{problem}
If $PN \parallel KF_2$, what will $e_1$, $e_2$, $\frac{ON}{OQ}$ and $\frac{OQ}{HQ}$ be?
\end{problem}
\begin{proposition}
If $PN \parallel KF_2$, then $e_1 = \frac{1}{\phi\sqrt{\phi}}, e_2 = \phi\sqrt{\phi}$ and $\frac{ON}{OQ}=\frac{OQ}{HQ}=\phi$, where $\phi=\frac{1 + \sqrt{5} }{2}$, in other words, $Q$ divides $ON$ into the golden ratio, and $H$ divides $OQ$ into the golden ratio.
\end{proposition}
\begin{proof}
First, solve the equation set (\ref{1})
\begin{equation}\label{1}
\left\{
\begin{array}{rl}
\frac{x^2}{a^2} + \frac{y^2}{b^2} = 1\\
x^2 - \frac{y^2}{b^2} = 1
\end{array}
\right.
\end{equation}
to get the coordinates of $P$, we get $x_P = OQ = \sqrt{\frac{2a^2}{a^2+1}}$, $y_P = PQ = \sqrt {\frac{{({a^2} - 1){b^2}}}{{{a^2} + 1}}}$, and by $PN \parallel KF_2$, we have $\frac{PQ}{QN}=\frac{KO}{OF_2}$, and get
\begin{equation}\label{1a} 
\sqrt {\frac{{({a^2} - 1){b^2}}}{{{a^2} + 1}}} = (a - \sqrt{\frac{2a^2}{a^2+1}})b
\end{equation}
and the final form
\begin{equation}\label{2} 
(a^2-1)(a^4 - 4a^2 - 1) = 0
\end{equation}
Since $c = 1$ and $a > c$, then we obtain the unique solution $a = \phi\sqrt{\phi}$ from (\ref{2}), hence $e_1 = \frac{c}{a} = 
\frac{1}{\phi\sqrt{\phi}}$, $e_2 = \phi\sqrt{\phi}$, then $\frac{ON}{OQ} = \frac{\phi\sqrt{\phi}}{\sqrt{\phi}} = \phi$, $OH =\frac{1}{a} = \frac{1}{\phi\sqrt{\phi}}$, $HQ = OQ-OH = \frac{1}{\sqrt{\phi}}$, thus, $\frac{OQ}{HQ} = \phi$, and we also get $QN = ON-OQ = \frac{1}{\sqrt{\phi}} = HQ$, which means $Q$ is the midpoint of $HN$.
\end{proof}
Next, we show an interesting property in the ellipse which has eccentricity $\frac{1}{\phi\sqrt{\phi}}$, let $AC$ and $BD$ denote the latus rectum of the ellipse, then we have
\begin{proposition}
The rectangle $ACDB$ is made up of 4 congruent right triangles similar to the Kepler triangle\cite[p. 149]{cte2} and $\frac{{F_1}{F_2}}{AF_1}=\sqrt{\phi}$. Also, it is shown in \cite{cte3} that, $\triangle{A{F_2}C}$ is just the kind of isosceles triangle of smallest perimeter which circumscribes a semicircle.
\end{proposition}
\begin{proof}
${F_1}{F_2} = 2$, ${AF_1} = \frac{2}{\sqrt{\phi}}$, then $\frac{{F_1}{F_2}}{AF_1}=\sqrt{\phi}$,
hence proved.
\end{proof}
Interestingly, we notice that the $\triangle{KOF_2}$ in Figure \ref{fig:fg1} is just the maximum generalized golden right triangle $T_2$\cite{cte4} which has sides $(1, \sqrt{2\phi}, \phi\sqrt{\phi})$, and we've already got an interesting construction in \cite{cte4}, next, we will show another two simple constructions of it, see Figure \ref{fig:fg32} and \ref{fig:fg33}.
\begin{figure}[ht]
\centering
\includegraphics[width=0.4\textwidth]{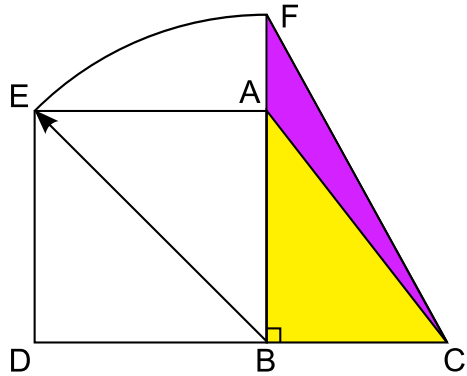}
\caption{A simple construction of $T_2$}
\label{fig:fg32}
\end{figure}
\begin{construction}\label{cons3}
A 3-step construction of $T_2$:
\begin{enumerate}[(1)]
\item construct a Kepler triangle $\triangle{ABC}$ with $BC = 1$, $AB = \sqrt{\phi}$ (see, e.g., \cite{cte5})
\item construct a square $ABDE$ externally on the side $AB$
\item draw an arc with the center at $B$ and the radius $BE$, cutting the extension of $BA$ at $F$, and join $F$ to $C$
\end{enumerate}
Then $\triangle{FBC}$ is $T_2$.
\end{construction}
\begin{proof}
$BF=BE=\sqrt{2}AB=\sqrt{2\phi}$.
\end{proof}
\begin{figure}[ht]
\centering
\includegraphics[width=0.5\textwidth]{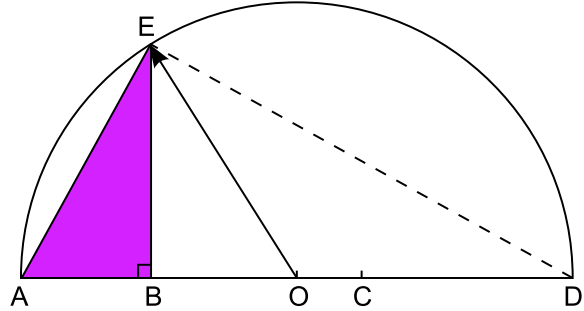}
\caption{Another simple construction of $T_2$}
\label{fig:fg33}
\end{figure}
\begin{construction}\label{cons4}
Another 3-step construction of $T_2$:
\begin{enumerate}[(1)]
\item construct a segment $AB = 1$, construct point $C$ on the extension of $AB$ that $\frac{BC}{BA}=\phi$, and extend $BC$ to $D$ that $CD=BC$
\item a semicircle is drawn with its center at the midpoint $O$ of $AD$, and the radius $OA$, naturally passing through $D$
\item through $B$, construct a perpendicular segment $BE$ to $AB$, and intersecting the semicircle at point $E$, and join $E$ to $A$
\end{enumerate}
Then $\triangle{ABE}$ is $T_2$.
\end{construction}
\begin{proof}
According to Thales' theorem, $\triangle{AED}$ is a right triangle, then we get $BE = \sqrt{{AB}\cdot{BD}} = \sqrt{2\phi}$.
\end{proof}



\medskip

\noindent Mathematics Subject Classification (2010).  51M04, 51M09, 51M15, 11B39

\noindent Keywords.  Conic section problem, Ellipse, Hyperbola, Golden ratio, Fibonacci numbers, Kepler triangle, Golden right triangle, Maximum generalized golden right triangle
\end{document}